\documentclass[12pt]{amsart}
\usepackage{color, graphicx, amsfonts, amsmath, amssymb, amsthm, amsrefs, amscd} 
\usepackage[colorlinks = false]{hyperref}
\usepackage{tikz-cd}
\usepackage{stmaryrd}

\usepackage{enumitem}

\calclayout

\newcommand{\C}{\mathbb{C}}
\newcommand{\N}{\mathbb{N}}
\newcommand{\Q}{\mathbb{Q}}
\newcommand{\R}{\mathbb{R}}

\newcommand{\holo}{\mathcal{O}}

\DeclareMathOperator{\aut}{Aut}
\DeclareMathOperator{\lie}{Lie}
\DeclareMathOperator{\vect}{VF}
\DeclareMathOperator{\sym}{Sym}

\theoremstyle{plain}
\newtheorem{theorem}{Theorem}[section]

\newtheorem{lemma}[theorem]{Lemma}
\newtheorem{proposition}[theorem]{Proposition}

\newtheorem{theorem:introduction}{Theorem}[]

\theoremstyle{definition}
\newtheorem{definition}[theorem]{Definition}
\newtheorem{remark}[theorem]{Remark}

\newcommand{\saut}{\mathrm{Aut}^\sigma}
\newcommand{\taut}{\mathrm{Aut}^\tau}

\DeclareMathOperator{\diff}{Diff}
\DeclareMathOperator{\id}{id}
\newcommand{\cotbun}{T^\ast M}
\newcommand{\rcotbun}{T^\ast_\mathbb{R} M_\mathbb{R}}
\newcommand{\ttau}{\tilde{\tau}}
\DeclareMathOperator{\Sym}{Sym}

\title[ Holomorphic approximation for Hamiltonian diffeomorphisms]{Holomorphic approximation for Hamiltonian diffeomorphism groups via the Symplectic Density Property}
\author{Gaofeng Huang}
\address{Department of Mathematics \\ Harvard University \\  
Cambridge MA, USA}
\email{ghuang@math.harvard.edu}

\author{Frank Kutzschebauch}
\address{Mathematical Institute \\ University of Bern \\  
Bern, Switzerland}
\email{frank.kutzschebauch@unibe.ch}

\begin{document}


\thanks{The first author was funded by Schweizerischer Nationalfonds Postdoc.Mobility [P500-$2\_239113$]. The second author was partially supported by Schweizerischer Nationalfonds [10005432]. }

\begin{abstract}
    We show that the cotangent bundle of a holomorphically flexible Stein manifold has the symplectic density property. As application we deduce holomorphic approximation for Hamiltonian diffeomorphisms of real cotangent bundles by complex symplectic automorphisms of the complexified cotangent bundles, which preserve the real bundle, in the Carleman sense, i.e.\@ in the Whitney topology. 
\end{abstract}

\maketitle

\setcounter{tocdepth}{1}
\tableofcontents

\section{Introduction}

Complexification of real objects and approximation of real functions by holomorphic functions is one of the most important tools in mathematics. For example a famous result of the Swedish mathematician Torsten Carleman states that a continuous function on the real line $\R$ can be approximated in the Whitney topology by holomorphic functions on the complex line $\C$. For an overview of such results, especially  in the setting of Several Complex Variables we refer to the overview article of John Erik Forn{\ae}ss, Franc Forstneri{\v c}  and Erlend Wold \cite{MR4264040}. In the Several Complex Variables community such approximation is commonly called Carleman approximation.

\smallskip
The step from approximating functions to approximating diffeomorphims is delicate.  For example approximation of  diffeomorphisms of the real line $\R$ by holomorphic automorphisms of $\C$ is clearly impossible. Diffeomorphism groups are always huge whereas the holomorphic automorphism group of $\C$ consists of affine maps only. So a necessary input for holomorphic approximation of diffeomorphism groups are complex manifolds with huge holomorphic automorphism groups. A theory about such complex manifolds, the so-called Anders\'en-Lempert theory, has been developed since about 35 years, see the overview article of Forstneri\v c and the second author \cite{MR4440754}. In this paper we concentrate on a symplectic version of that theory, which as remarked in \cite{MR4440754} (before this article) was still lacking criteria and examples. 
\medskip

Let $X$ be a smooth complex affine algebraic variety equipped with an algebraic symplectic form $\omega$. Denote by $O(X)$ the algebra of regular functions on $X$. With the Poisson bracket induced by the symplectic form, $O(X)$ is an infinite dimensional Lie algebra. Furthermore, denote by $IHF_X$ the subspace of $O(X)$ consisting of regular functions associated to complete (globally integrable) Hamiltonian vector fields: 
\begin{align*}
    IHF_X = \{\, f \in O(X) : \text{ the Hamiltonian vector field } V_f \text{ is complete} \,\}.
\end{align*}
The Lie subalgebra generated by $IHF_X$ is denoted as $\lie_X$:
\begin{align*}
    \lie_X := \lie(IHF_X).  
\end{align*}

\begin{definition}
    We say that a smooth symplectic affine variety $(X, \omega)$ has the \textbf{algebraic symplectic density property} if the Lie algebra generated by complete algebraic symplectic vector fields coincides with the Lie algebra of algebraic symplectic vector fields; similarly, $(X, \omega)$ has the \textbf{algebraic Hamiltonian density property} if the Lie algebra generated by complete algebraic Hamiltonian vector fields coincides with the Lie algebra of algebraic Hamiltonian vector fields; or equivalently,
    \[
        \lie_X = \C[X] \mod \C.
    \]
    We say that a smooth affine variety $M$ is \textbf{flexible}, if locally nilpotent algebraic vector fields (LNDs) on $M$ span the tangent space $T_x M$ at every point $x \in M$.
\end{definition}

A main achievement of the present paper is to provide the first big class of Stein manifolds with the symplectic density property, a fundamental class in symplectic geometry, the class of  cotangent bundles. We start with an algebraic version of our results, since many algebraic methods are in the core of our arguments.

\begin{theorem:introduction} \label{thm: algHDP}
    The algebraic cotangent bundle of every smooth flexible  complex affine algebraic variety has the algebraic symplectic density property and the algebraic Hamiltonian density property. 
\end{theorem:introduction}

There is an abundance of affine flexible manifolds, the study of which has become one of the most active topics in Affine Algebraic Geometry. A well-studied class of algebraic manifolds is among them: the affine homogeneous spaces.

\bigskip 

Here comes  the (more general) holomorphic counterpart of the above algebraic definitions and results.

Let $Y$ be a Stein manifold equipped with a holomorphic symplectic form $\Omega$ and $\holo(Y)$ the algebra of holomorphic functions on $Y$. Denote by $IHF_Y$ the subset of $\holo(Y)$ consisting of holomorphic functions associated to complete Hamiltonian vector fields: 
\begin{align*}
    IHF_Y = \{\, f \in \holo(Y) : \text{ the Hamiltonian vector field } V_f \text{ is complete} \,\}.
\end{align*}
The Lie algebra generated by $IHF_Y$ is denoted as $\lie_Y$:
\begin{align*}
    \lie_Y := \lie(IHF_Y).  
\end{align*}

\begin{definition}
    We say that a Stein symplectic manifold $(Y, \Omega)$ has the \textbf{symplectic density property} if the Lie algebra generated by complete holomorphic symplectic vector fields is dense in the Lie algebra of holomorphic symplectic vector fields, in the compact-open topology; similarly, $(Y, \Omega)$ has the \textbf{Hamiltonian density property} if the Lie algebra generated by complete holomorphic Hamiltonian vector fields is dense in the Lie algebra of holomorphic Hamiltonian vector fields; or equivalently, 
    \[
        \overline{\lie_Y} = \holo(Y) \mod \C.
    \]
    We say that a Stein manifold $M$ is \textbf{holomorphically flexible}, if the complete holomorphic vector fields on $M$ span the tangent space $T_x M$ at every point $x \in M$.
\end{definition}

\begin{theorem:introduction} \label{thm: holoHDP}
    The holomorphic cotangent bundle of every holomorphically flexible Stein manifold has the symplectic density property and the Hamiltonian density property. 
    
    In particular, the holomorphic cotangent bundle of every Stein manifold with the density property has the symplectic density property.
\end{theorem:introduction}

Having these symplectic manifolds with the symplectic density property, i.e.\@ with huge groups of holomorphic symplectomorphisms, to our disposal we can use them to  approximate groups of Hamiltonian diffeomorphisms of real symplectic manifolds. First let us point out that approximation is meant in the following strongest possible sense. 

\begin{definition}
    Let $X_\R$ be a smooth manifold and $G_\R$ a subgroup of $\diff(X_\R)$. We say that $G_\R$ is \textbf{holomorphically approximable}, if there exist a complexification $(X, \sigma)$ of $X_\R$ and a subgroup $G$ of $\saut (X)$, such that $G \rvert_{X_\R} \subset G_\R$ and $G \rvert_{X_\R}$ is dense in $G_\R$ in the Whitney topology.  
\end{definition}

Our main holomorphic approximation  result is the following, see Theorem \ref{Ham-Carleman}. For the notion of $\tau$-holomorphic flexibility, see Definition \ref{tau-hol-flex}.
 
\begin{theorem:introduction} \label{thm-Carleman}
    If a smooth symplectic manifold $M_\R$ has a complexification $(M, \tau)$ which is Stein and $\tau$-holomorphically flexible, then the group of Hamiltonian diffeomorphisms of the cotangent bundle  $T^* M_\R$ is holomorphically approximable. 
\end{theorem:introduction}

As examples we prove that homogeneous spaces $G_\R/H_\R$ and the Danielewski surfaces admit flexible complexifications, see Section \ref{sec: realFlex}.

\smallskip
In classical mechanics, every physical system has a configuration space and the cotangent bundle over the configuration space is the phase space of the system. When the base manifold is a homogeneous space, the forward time evolution, i.e.\@ the Hamiltonian flow of the energy function, are globally approximable by holomorphic Hamiltonian flows to any finite order. This illustrates that symmetries of these systems are preserved by complexification.

\medskip 

In another recent paper \cite{DHKW26} we have more systematically developed the application of density property, volume density property and symplectic density property to the approximation of diffeomorphism groups, groups of volume preserving diffeomorphisms and symplectic diffeomorphisms.  There already existed some isolated results in this direction:  approximation of Hamiltonian diffeomorphisms of $(S^1)^{2n}$, $(S^1)^n \times \R^n$ and $\R^{2n}$ (however in the much coarser compact-open topology and in the case $(S^1)^n \times \R^n \cong T^*(S^1)^n$ for another symplectic form than the canonical form on $T^*(S^1)^n$) by Pierre Berger and Dmitry Turaev \cite{MR4930686}, approximation of Hamiltonian diffeomorphisms of coadjoint orbits by Fusheng Deng and Erlend Wold \cite{MR4423269} as well as symplectic diffeomorphisms of the real Calogero-Moser space by the first author \cite{CaloCarleman}.  

\medskip

The paper is structured as follows. We first consider some properties of the Hamiltonian vector fields and the structure of the Hamiltonian functions on the cotangent bundle in Section \ref{sec: HVF}. Then in Section \ref{sec: KK-Method} we give the main arguments for the symplectic density property. Next in Section \ref{sec: Ham-Carleman} we show that a variant of symplectic density property implies the holomorphic approximability of the Hamiltonian diffeomorphism group of the cotangent bundle. Finally as application, we prove in Section \ref{sec: realFlex} this symplectic density property for affine homogeneous spaces of linear algebraic groups and smooth Danielewski surfaces. 

\section{Hamiltonian Vector Fields on the Cotangent Bundle} \label{sec: HVF}

\subsection{Cotangent Liftings}
Let $M$ be a Stein manifold (or a smooth algebraic variety) of dimension $n$; let $\pi \colon T^*M \to M$ be the holomorphic (resp.\ algebraic) cotangent bundle over $M$. For $(q,p) \in \cotbun$, $v \in T_{(q,p)}(\cotbun)$, the tautological 1-form $\alpha$ is defined by $\alpha(v) = p ( d \pi_{(q,p)} (v) )$.
Also, let $\omega = -d \alpha$ be the canonical symplectic form on $T^*M$. In local coordinates $(q^1, \dots, q^n)$, let $p_1, \dots, p_n$ be fiberwise linear coordinates on $T^*M$. Then 
\begin{align*} 
    \alpha = \sum_i p_i d q^i, \quad \omega = -d \alpha = \sum_i d q^i \wedge d p_i.
\end{align*}

Given a function $f$ on $M$ we obtain a function $\pi^* f$ on $T^*M$: 
\begin{align*}
    \holo(M) \hookrightarrow \holo(T^*M), \quad  f \mapsto \pi^* f.
\end{align*}

\begin{definition}
    Let $\beta$ be an algebraic (or holomorphic) 1-form on $M$. The pullback $\pi^*\beta$ determines a vector field $V_\beta$ on $T^*M$ via 
    \begin{align} \label{lift-1-form-sympl}
        i_{V_\beta} \omega =  \pi^*\beta .
    \end{align}
    We call $V_\beta$ the {\bf vertical lift} of the 1-form $\beta$.
\end{definition}

\begin{lemma} \label{lem: fctsonM}
    For a 1-form $\beta \in \Omega^1(M)$, its vertical lift $V_\beta$ given by $i_{V_\beta} \omega = \pi^* \beta$ is a complete vector field on $T^*M$. Moreover, if $\beta$ is closed (or exact), $V_\beta$ is symplectic (Hamiltonian). In particular, any regular function on $M$ induces a complete Hamiltonian vector field on the cotangent bundle $T^*M$, i.e.,
    \[
        \pi^* \holo(M) \subseteq IHF_{\cotbun}.
    \]
\end{lemma}
\begin{proof}
    In local coordinates, we have $\beta = \sum_i \beta_i d q^i$ and 
    \begin{align} \label{Lift-1form}
        V_\beta = 0 \, \partial_{q^1} + \cdots + 0 \, \partial_{q^n} - \beta_1 \partial_{p_1} - \cdots - \beta_n \partial_{p_n}
    \end{align}
    which is vertical, i.e.\@ $\pi_* V_\beta = 0$. On every fiber $\pi^{-1}(q)$, the restriction $V_\beta \rvert_{\pi^{-1}(q)}$ is constant and hence complete.

    When $\beta$ is closed, it follows from $i_{V_\beta} \omega = \pi^* \beta$ that $V_\beta$ is symplectic.

    When $\beta$ is exact, $\beta = df$ for some regular function $f$ on $M$. 
    Since $i_{V_\beta} \omega = \pi^* df = d (\pi^* f)$, the claim follows. Note that $V_{df}$ is the Hamiltonian vector field associated to $\pi^* f$ on $T^*M$. 
\end{proof}

\begin{definition}
    Let $V$ be a  vector field on the base $M$. 
    We denote by 
    \begin{align} \label{V-h_V}
        h_V (q, p) :  = p(V)(q) = \sum_i p_i V^i (q) \, \text{ where } (q, p) \in T^*M,
    \end{align}
    the fiber-wise linear function on $T^*M$ induced by $V$. The {\bf Hamiltonian lift} of $V$ is the Hamiltonian vector field on $T^*M$ associated to $h_V$.
\end{definition}
This map sending vector fields to functions on $T^*M$ is $\holo(M)$-linear:
\begin{align} \label{hfV = fhV}
    h_{f V} = (\pi^*f) \, h_V \, \text{ for any } f \in \holo(M).
\end{align}

\begin{lemma} \label{lem: Vliftcomplete}
    Let $V$ be a vector field on the base manifold $M$. If $V$ is complete, then its Hamiltonian lift $\hat{V}$ to the cotangent bundle $T^*M$ is also complete. 
\end{lemma}
\begin{proof}
    The Hamiltonian lift $\hat{V}$ of $V$ is given locally by
    \begin{align} \label{Ham-lift}
        (V^1, \dots, V^n, - \sum_a p_a \partial_{x^1} V^a, \dots, - \sum_a p_a \partial_{x^n} V^a ).
    \end{align}
    In the vertical directions, the lift is linear in $p_a$'s. Thus, the completeness of the horizontal component implies the completeness of the lift. 
    
    Note that when $V$ is complete, its flow map $\varphi_t \in \aut(M), t \in \C$ induces a transformation on $T^*M$ as $(q, p) \mapsto ( \varphi_t(q), (\varphi_{-t})_q^* p(q) )$, which is the flow map of the lift. 
\end{proof}

\begin{lemma} \label{mix-field-function}
    Let $V$ be a vector field on the base manifold $M$. If $V$ is complete, then 
    \[
         \pi^* \holo(M) [h_V] = \{\, (\pi^* f) h_V^k: f \in O(M), k \in \N \,\}  \subseteq \lie_{T^*M}.
    \]
\end{lemma}
\begin{proof}
    Let $f \in \holo(M)$. From \eqref{Lift-1form} and \eqref{Ham-lift}, the Hamiltonian vector field on $T^*M$ associated to $\pi^* f + h_V$ is locally given by
    \begin{align*}
        V^1 \partial_{q^1} + \cdots + V^n  \partial_{q^n} - \left( \partial_{q^1}f + \sum_a p_a \partial_{q^1} V^a \right) \partial_{p_1} - \cdots - \left( \partial_{q^n}f + \sum_a p_a \partial_{q^n} V^a \right) \partial_{p_n}
    \end{align*}
    In the base direction the flow is complete by assumption, whereas in each linear fiber the vector field is linear in the fiber coordinates $p_a$ and thus integrable for all time. Thus $\pi^* f + h_V  \in IHF_{T^*M}$ which implies $(\lambda \pi^*f + h_V)^{k+1} \in IHF_{T^*M}$ for any $\lambda \in \C$ and $k \in \N$. Using pairwise distinct scalars $\lambda_1, \lambda_2, \dots, \lambda_{k+2}$ and the invertibility of the Vandermonde coefficient matrix, we get $(\pi^* f) h_V^k$ as a linear combination of $(\lambda_i \pi^*f + h_V)^{k+1}$ and thus an element of $\lie_{T^* M}$. 
\end{proof}

\begin{remark}
    We use the sign convention $\{ F, G \} = \omega( V_G, V_F )$ where $F,G \in \holo(T^*M)$. 
    The following Poisson brackets associated to the canonical symplectic form $\omega$ on $T^*M$ are practical (Kentaro Yano and Edward Patterson \cite{MR0206868}, Propositions 4, 7 and 9):
    \begin{align*}
        \{ \pi^*f, \pi^*g \} = 0, \quad
        \{ h_V, \pi^*f \} = V(f), \quad
        \{ h_V, h_W \} = h_{[V,W]}.
    \end{align*}
    where $f,g \in \holo(M)$, $V, W$ holomorphic vector fields on $M$ and $h_V$ the Hamiltonian function on $T^*M$ associated to the Hamiltonian lift of $V$. 
\end{remark}

\begin{lemma} \label{lem: C[h_V]}
    If $f \in IHF_X$, then $\C[f] \subset IHF_X$. 
\end{lemma}
\begin{proof}
    Since the Hamiltonian vector field $V_f$ is complete, $(i+1)f^i V_f$ is also complete and Hamiltonian with function $f^{i+1}$. The same argument works for a polynomial in $f$. 
\end{proof}

\begin{lemma} \label{lem: fgIHF}
    Let $f, g \in IHF_X$. If $\{f, g\}$ is constant, then $f+g, fg \in IHF_X$.
\end{lemma}
\begin{proof}
    Let $V, W$ be the Hamiltonian vector fields associated to $f$ and $g$, respectively. From the assumption, $[V,W]=0$. Then since $V, W$ are complete, $V+W$ is complete as well. Thus $f+g \in IHF_X$ and also $(f+g)^2 \in IHF_X$ by Lemma \ref{lem: C[h_V]}. 
    
    The Hamiltonian vector field of $fg$ is $gV + fW$. Since $V(g) = - W(f) = -\{ f, g \}$ is constant, we have that $gV$ and $fW$ are complete. Finally, the vanishing of their commutator $[gV, fW]$ yields the complete integrability of their sum. 
\end{proof}

\subsection{Hamiltonian functions on the cotangent bundle}

Let $M$ be a complex manifold of dimension $n$, $\holo_M$ the sheaf of holomorphic functions on $M$ and $\mathcal{T}_M$ the tangent sheaf of $M$. Denote by $\vect(M) = \Gamma(M, \mathcal{T}_M)$ the $\holo_M$-module of holomorphic vector fields on $M$.

The $k$-th symmetric power of the tangent sheaf $\mathcal{T}_M$, denoted by $\Sym^k \mathcal{T}_M$, is locally free of rank $\binom{n+k-1}{k}$ and thus coherent.
For a vector space $E$, the $\C$-algebra $\Sym^k E$ of $k$-th symmetric tensors  is isomorphic to the $\C$-algebra $\C[E^*]_k$ of homogeneous polynomials of degree $k$ on the dual space $E^*$:
\begin{align} \label{symTensor<->homogPolyn}
    \sigma\colon \sym^k E \to \C[E^*]_k, \quad \sigma(v_1 v_2 \cdots v_k)(l) := l(v_1) l(v_2)\cdots l(v_k) \quad \forall \, l \in E^*, 
\end{align}
which has the $k$-th derivative at the origin $0 \in E^*$ as its inverse. Applying this isomorphism fiberwise to the tangent sheaf we obtain the morphism of sheaves of $\holo_M$-modules
\begin{align*}
    \sigma \colon \sym^k \mathcal{T}_M \to \pi_* \holo_{T^*M}, \quad \sigma(S)(q,p) := \langle S(q), p^{\otimes k} \rangle \quad \forall (q,p)\in T^*M 
\end{align*}
which is onto the subsheaf of fiberwise homogeneous polynomials of degree $k$ with $\holo(M)$ coefficients. 
This gives a description for 
\begin{align*}
    \Gamma(M, \sym^k \mathcal{T}_M) &\cong \{ f \in \holo(T^*M) : f \text{ homogeneous of degree $k$ on fibers} \} \\
    &= \{ f \in \holo(T^*M): \Theta f = k f \} \text{ where $\Theta$ is the Euler vector field}.
\end{align*}

The projection from $\holo(T^*M)$ to the $k$-th homogeneous component can be written as a Fourier projection along the fiber:
\[
f_k(q,p)=\frac1{2\pi i}\oint_{|\lambda|=1}f(q,\lambda p)\,\lambda^{-k-1}\,d\lambda \quad \text{ for } f \in\holo(T^*M),
\]
which is the coefficient in the Taylor expansion $f(q, \lambda p) = \sum f_k \lambda^k$. By Cauchy's estimate the sum of homogeneous components converges uniformly on compact subsets, therefore 
\begin{align} \label{dense-in-O(T*M)}
    \bigoplus_{k}\Gamma(M, \sym^k \mathcal{T}_M)  \text{ is dense in } \holo(T^*M).
\end{align}
With this identification in mind, we consider the filtered sheaf on $M$
\begin{align*}
    \mathcal{F}_k = \bigoplus_{j=0}^k \Sym^j \mathcal{T}_M \subset \pi_* \holo_{T^*M},
\end{align*}
which is coherent as being a direct sum of coherent sheaves. At the two lowest degrees 
\begin{align*}
    \holo(T^*M)_0 = \pi^* \holo(M), \quad \holo(T^*M)_1 = \{\, h_V : V \in \vect(M) \, \}. 
\end{align*}
With respect to the Poisson bracket of the canonical symplectic form on $T^*M$ the filtration behaves as
\begin{align} \label{filtration-bracket}
    \{ \, \holo(T^*M)_j, \holo(T^*M)_k \} \subset \holo(T^*M)_{j+k-1}
\end{align}
Furthermore $\pi_* \holo_{T^*M}$ can be seen as the sheaf of graded algebras of differential operators on $M$, see Hotta, Takeuchi and Tanisaki \cite{MR2357361}*{\S 1.1}.

\section{Proof of the Symplectic Density Property} \label{sec: KK-Method}

We explain the main steps in proving Theorem \ref{thm: algHDP} and Theorem \ref{thm: holoHDP} using notations adapted to the Stein setting, while most arguments apply verbatim to the affine case. 

We may assume $\dim M \ge 2$. 
In dimension 1, the only smooth flexible affine variety is $\C$ and the holomorphically flexible Stein manifolds are $\C$ or $\C^*$. Their cotangent bundles have the symplectic and the Hamiltonian density property, see Forstneri{\v c} \cite{MR1408866}*{\S 5}, Dror Varolin \cite{MR1829353} and Rafael Andrist--Huang \cite{MR4986786}*{Lemma 3.2, Theorem 3.5}.
\medskip

First, we reduce the symplectic case to a Hamiltonian one. 

\begin{lemma} \label{SDP-to-HDP}
    If the cotangent bundle $T^* M$ has the Hamiltonian density property, then it has the symplectic density property. 
\end{lemma}
\begin{proof}
When the cotangent bundle has trivial first De Rham cohomology, we are in the Hamiltonian case. If not, let $V$ be a symplectic vector field on $\cotbun$ which is not Hamiltonian. Then the 1-form $i_V \omega$ represents a nontrivial class in $H^1_{DR}(T^*M)$. 
Since the total space $T^*M$ is homotopy equivalent to the base $M$, there is a closed 1-form $\alpha$ on $M$ representing the corresponding class in $H^1_{DR}(M)$. 
Then through the correspondence \eqref{lift-1-form-sympl}, we obtain a symplectic, non-Hamiltonian vector field $W$ on $T^*M$, which is the vertical lift of $\alpha$ and thus complete. Then the difference $V-W$ is Hamiltonian and can be written as a Lie combination of complete Hamiltonian vector fields by assumption. 
\end{proof}

\medskip

By the holomorphic flexibility of $M$ there exist finitely many complete vector fields $V_1, \dots, V_N$ which span the tangent bundle $TM$. Thus $\sym^k TM$ can be spanned by symmetric tensors of degree $k$ in $V_1, \dots, V_N$, which by the correspondence \eqref{symTensor<->homogPolyn} are monomials of degree $k$ in $h_{V_1}, \dots, h_{V_N}$. Hence
\begin{align} \label{deg-k-surjection}
    \holo(M) [h_{V_1}, \dots, h_{V_N}]_k \twoheadrightarrow \holo(T^*M)_k \cong \Gamma(M, \sym^k \mathcal{T}_M). 
\end{align}

By Cartan's Theorem B, for any holomorphic vector field $V$ on $M$, there exist functions $f_1, f_2, \dots, f_N$ in $\holo(M)$ such that
\begin{align} \label{VF-sum}
    V = f_1 V_1 + f_2 V_2 + \cdots + f_N V_N.
\end{align}
By Lemma \ref{mix-field-function} we have  $\holo(M)[h_{V_i}] \subset \lie_X$. Hence 
\begin{align} \label{h_V-in-LieX}
    h_V = \sum_{i=1}^N f_i h_{V_i} \in \lie_X \text { for any vector field } V \text{ on } M. 
\end{align}

For the Hamiltonian density property of $X=T^*M$, it suffices to show 
\begin{align} \label{verifyHDP}
    \holo(M) [h_{V_1},h_{V_2},\dots, h_{V_N}] \subset \lie_X
\end{align}
since $\holo(M) [h_{V_1},\dots, h_{V_N}]$ is dense in $\holo(T^*M)$ by \eqref{dense-in-O(T*M)} and \eqref{deg-k-surjection}.

\medskip

Fix $k \in \N$ and consider $\mathcal{F}_k = \oplus_{j \le k} \Sym^j \mathcal{T}_M$.

\begin{lemma} \label{Stalk-at-pt}
    There exist a subsheaf $\mathcal{E}_k \subset \mathcal{F}_k$ of $\holo_M$-modules, such that all global sections of $\mathcal{E}_k$ are in $\lie_X$ and the stalks $(\mathcal{E}_k)_q = (\mathcal{F}_k)_q$ are identical for every $q \in M$.
\end{lemma}

\begin{proof}
    From Lemma \ref{mix-field-function} we have $\holo(M)h_{V}^k \subset \lie_X$ for every complete vector field $V$ and by \eqref{h_V-in-LieX} $h_W \in \lie_X$ for any vector field $W$ on $M$. 
    We take $\mathcal{E}_k$ to be the smallest linear subspace of $\mathcal{F}_{k}$ containing 
    \begin{align*}
        \{ \, f h_{V}^k: \, f \in \holo(M),  \, V \in \vect(M) \text{ complete} \,\} \subset \lie_X
    \end{align*}
    and stable under Lie derivatives $L_V$ with respect to any vector field $V$ on $M$. Note that for $f \in \holo(T^*M)$ one has $L_V f = \{ h_V, f \}$. 

    We claim that $\mathcal{E}_k$ is an $\holo_M$-module. Indeed we can write $\mathcal{E}_k$ as the union of subsets graded by iterative action of the Lie derivative
    \[
        \mathcal{E}_k = \bigcup_{l=0}^\infty \mathcal{E}_k^l
    \]
    where $\mathcal{E}_k^0 = \{ \, f h_{V}^j: \, f \in \holo(M),  \, V \in \vect(M) \text{ complete}, j \le k \, \}$ and 
    \[
        \mathcal{E}_k^{l+1} = \mathcal{E}_k^l + \{\, L_V s: V\in \vect(M), s \in \mathcal{E}_k^l \,\}.
    \]
    At the induction base $\mathcal{E}_k^0$ is an $\holo_M$-module by Lemma \ref{mix-field-function}. Suppose $\mathcal{E}_k^l$ is an $\holo_M$-module. Let $f,g \in \holo(M)$, $V \in \vect(M)$ and $s \in \mathcal{E}_k^l$ one has
    \[
        f L_V s = L_V( f s) - V (f) s \in \mathcal{E}_k^{l+1}
    \]
    which shows $\mathcal{E}_k^{l+1}$ is an $\holo_M$-module.

    By definition and \eqref{h_V-in-LieX} all global sections of $\mathcal{E}_k$ are in $\lie_X$.
    
    Now let $j \le k$. At $q \in M$ the evaluation
    \[
        \mathcal{E}_k(q) \cap \sym^j(T_qM) \subset \sym^j(T_qM)
    \]
    is by definition stable under $\mathfrak{gl}(T_qM)$: for any $A \in \mathfrak{gl}(T_qM)$ there exists a holomorphic vector field $V$ on $M$ such that $V(q)=0$, $DV(q) = A$ thus $L_V s = A s$ for $s \in \mathcal{E}_k(q) \cap \sym^j(T_qM)$. On a Stein manifold, we can prescribe the first jet of a holomorphic vector field at a point using the coherence of $\vect(M)$ and Cartan's Theorem A. From the irreducibility of $\sym^j(T_qM)$ as a $\mathfrak{gl}_n$-module (see Fulton-Harris \cite{MR1153249}*{Theorem 6.3}) and the fact that $n$ vectors from $V_1(q), V_2(q), \dots, V_N(q)$ span $T_qM$, it follows 
    \[
        \mathcal{E}_k(q) = \bigoplus_{j \le k} \sym^j (T_qM) = \mathcal{F}_k(q).
    \]
    By Nakayama's lemma, the stalks of $\mathcal{E}_k$ and $\mathcal{F}_k$ at $q$ are identical. 
\end{proof}

\begin{lemma} \label{stalks-to-sheaf}
    There exist finitely many global sections $s_1, s_2, \dots, s_m$ of $\mathcal{E}_k \subset \mathcal{F}_k$ such that the germs $(s_i)_q$ generate the stalk $(\mathcal{F}_k)_q$ at every $q \in M$.
\end{lemma}
\begin{proof}
    We start at one point $q$ and a set of global sections of $\mathcal{E}_k$ spanning the fiber $\mathcal{E}_k(q)= \oplus_{j \le k} \sym^j (T_qM)$. Let $A$ be the analytic subvariety where these sections do not span. Take $q' \in A$. Since $M$ is holomorphically flexible and $\dim M \ge 2$, the group of automorphisms $\aut(M)$ acts transitively on $M$. Thus, there is an automorphism $\psi$ taking $q'$ to a point $\psi(q') \in M \setminus A$ where the existing sections span. Then adding  $\psi^* s_i$ to the set of sections we get spanning sections at $q'$. Note that $\mathcal{E}_k$ is invariant under pullback by automorphisms.

The points of $M$ where the sections still fail to span is a proper subset $B$ of lower dimension in $M$. On $B$ we can repeat the above process. In the affine case, the transitivity of $\aut(M)$ and an induction on the dimension of $B$ give global sections spanning everywhere. In the Stein case, the analytic subset might have countably many irreducible components $B_1, B_2, \dots$. Since $\aut(M)$ is a complete metric space, there exists by the second author \cite{MR3229363}*{Lemma 25} a holomorphic automorphism $\phi$ such that $\phi(M \setminus B) \cap B_i \neq \emptyset$ for every $i = 1, 2, \dots$. 

Therefore finitely many sections of $\mathcal{E}_k$ generate the stalks of $\mathcal{F}_k$ everywhere.
\end{proof}

This induces a morphism of sheaves $\holo_M^m \to \mathcal{F}_k$ which is surjective at the stalks. 
A standard application of Theorem B to the $\holo_M$-module $\mathcal{F}_k$ implies the global sections of the two sheaves $\mathcal{E}_k$, $\mathcal{F}_k$ are identical and thus all global sections of $\mathcal{F}_k$ are in $\lie_X$. Since this holds for arbitrary $k \in \N$, we obtain $\holo(M)[h_{V_1}, \dots, h_{V_N}] \subset \lie_{X}$.  
By the observation \eqref{verifyHDP}, this finishes the proof of Theorem \ref{thm: algHDP} and Theorem \ref{thm: holoHDP}.

\smallskip

\section{Holomorphic Approximation for Hamiltonian Diffeomorphisms} \label{sec: Ham-Carleman}

Let $X$ and $Y$ be smooth manifolds and $p \in X$. Consider smooth maps $f, g \colon X \to Y$ with $f(p)=g(p)=q$. We say that $f$ has first order contact with $g$ at $p$ if $df_p = dg_p$ as maps of $T_pX \to T_q Y$; and $f$ has $k$th order contact with $g$ at $p$ if $df \colon TX \to TY$ has $(k-1)$th order contact with $dg$ at every point in $T_p X$. This gives an equivalence relation for smooth maps in $C^\infty(X,Y)$ for each positive integer $k$. Denote by $J^k(X,Y)_{p,q}$ the set of equivalence classes of maps $f \colon X \to Y$ with $f(p)=q$. The disjoint union $J^k(X,Y) := \cup_{(p,q) \in X \times Y} J^k(X,Y)_{p,q}$ is a smooth manifold. 

Every smooth map $f \colon X \to Y$ induces a smooth map $j^kf \colon X \to J^k(X,Y)$ defined by 
\[
    j^kf(p)= [f] \in J^k(X,Y)_{p,f(p)} \text{ for every } p \in X. 
\]
The family of sets $\{ f \in C^\infty(X,Y): j^kf(X) \subset U \}$ where $U$ is an open subset of $J^k(X,Y)$ form a basis for a topology on $C^\infty(X,Y)$, which is called the Whitney $C^k$-topology. By choosing Riemannian metrics on the manifolds $X$ and $Y$ one can equip the fibers of the jet space with norms, i.e. the jet space $J^k(X,Y)$ with a Finsler metric. Over a compact $X$ the jet space $J^k(X,Y)$  is then equipped with
a norm (the $C^k$-norm) compatible with the the Whitney $C^k$-topology. By using a sequence of exhausting compacts of $X$  one gets a metric $d_k$ on the manifold $J^k(X,Y)$ compatible with this topology.

\begin{definition}\label{definition: realform}
(1) Let $M_\R$ be a smooth manifold of dimension $n \ge 1$. A \textbf{complexification} $(M, \tau, \jmath)$ of $M_\R$ is a complex manifold $M$ of complex dimension $n$, together with an antiholomorphic involution $\tau \colon M \to M$ and a smooth map $\jmath \colon M_\R \to M$, such that $\jmath$ is a proper embedding and the image of $M_\R$ is the fixed-point set of $\tau$.

\smallskip
(2) [Deng-Wold \cite{MR4423269}]
Let $(X_\R, \omega_\R)$ be a smooth symplectic manifold of dimension $2n, \, n \ge 1$. A \textbf{symplectic complexification} $(X, \omega, \sigma)$ of $(X_\R, \omega_\R)$ is a holomorphic symplectic manifold $X$ of complex dimension $2n$, together with a holomorphic symplectic form $\omega$, an antiholomorphic involution $\sigma \colon X \to X$ and a smooth map $\jmath \colon X_\R \to X$, satisfying
\begin{enumerate}[label=(\roman*)]
	\item		$(X, \sigma)$ is a complexification of $X_\R$.
	\item		The pullback $\jmath^\ast \omega$ coincides with $\omega_\R$. 	
	\item		The pullback $\sigma^\ast \omega$ is the complex-conjugated $\overline{\omega}$.
\end{enumerate}
\end{definition}

Let $M$ be a Stein manifold and $\tau \colon M \to M$ an antiholomorphic involution. 
The involution $\tau$ on $M$ induces an  involutive fiberwise $\C$-linear bundle map on the complexified cotangent bundle $T_\R^*M \otimes \C$
\begin{align*}
    (d \tau)^* \colon T_\R^*M \otimes \C \to T_\R^*M \otimes \C,
\end{align*}
which maps $\cotbun = (T^{(1,0)} M)^*$ to $(T^{(0,1)} M)^*$ and vice versa. Composition with complex conjugation
\begin{align*}
    T_\R^*M \otimes \C \to T_\R^*M \otimes \C, (x, p_x) \mapsto (\tau(x), \overline{(d_{\tau(x)} \tau)^* (p_x )})
\end{align*}
restricts to an antiholomorphic involution on the holomorphic cotangent bundle
\begin{align*}
    \ttau: \cotbun  \to \cotbun, (x, p_x) \mapsto (\tau(x), \overline{(d_{\tau(x)} \tau)^* (p_x )}).  
\end{align*}

Denote by $M_\R = \mathrm{Fix}(\tau)$ the set of fixed points of $\tau$. We assume $M_\R \neq \emptyset$, then $M_\R$ is a totally real submanifold of maximal dimension. Indeed, by Aldo Andreotti and Per Holm \cite{MR0589903}*{Proposition 1.3} there exist holomorphic coordinates $z_1, z_2, \dots, z_{n}$ around any point $x \in M_\R$ such that $\tau(z_i) = \overline{z_i}$ for $i= 1, \dots, n$. 
Since $\R$-linear functionals on $T^*_{\R,x} M_\R$ can be considered as $\C$-linear functionals on $T^*_x M$, $\rcotbun$ is a subbundle of $\cotbun \rvert_{M_\R}$ given as the fixed-point set of $\ttau$. 

\begin{lemma} \label{T^*M-SympCompl}
    $(\cotbun, \omega, \ttau)$ is a symplectic complexification of $(\rcotbun, \omega_\R)$. 
\end{lemma}
\begin{proof}
    Let $\pi \colon \cotbun \to M$ be the holomorphic cotangent bundle. For $(x,p) \in \cotbun$, $v \in T_{(x,p)}(\cotbun)$, the tautological 1-form $\alpha$ is defined by $\alpha(v) = p ( d \pi_{(x,p)} (v) )$. 

    Similarly, for the real cotangent bundle $\pi_\R \colon \rcotbun \to M_\R$, the tautological 1-form $\alpha_\R$ is given by $\alpha_\R (v) = p ( d \pi_{\R, (x,p)} (v))$ for $(x,p) \in \rcotbun, v \in T_{(x,p)}^\R(\rcotbun)$. 
    
    Since $\pi \circ \ttau = \tau \circ \pi$, the following diagram commutes
    \[
    \begin{CD}
	   \rcotbun @>\jmath>>	 \cotbun \\
	   @V\pi_\R VV 			 @VV \pi V\\
	   M_\R	@>i>>	M
    \end{CD}
    \]
    Thus $\jmath^* \pi = \pi_\R$ and $\jmath^* \alpha = \alpha_\R$. Since $\omega= -d \alpha$ and $\omega_\R= - d \alpha_\R$, we have $\jmath^* \omega = \omega_\R$. 

    For $v \in T_{(x,p)}(\cotbun)$
    \begin{align*}
        \ttau^* \alpha (v) &= \alpha (d \ttau (v) ) = \overline{ (d \tau)^* p (d \pi \circ d \ttau (v)) } = \overline{ (d \tau)^* p (d \tau \circ d \pi (v)) } \\
        &= \overline{ p(d \pi (v))} = \overline{ \alpha(v)}. 
    \end{align*}
    It follows that $\ttau^* \omega = \overline{ \omega}$. 
\end{proof}

\begin{definition}
    Let $X$ be a complex manifold with an antiholomorphic involution $\sigma \colon X \to X$.  
    A holomorphic vector field $V$ on $X$ is called $\sigma$-\textbf{compatible} 
    if $\sigma_* (V) = \overline{V}$. 
    A holomorphic function $f \in \holo(X)$ is called $\sigma$-\textbf{compatible}  if $\sigma^\ast f= \overline{f}$ and likewise, a holomorphic differential form $\alpha$ is $\sigma$-\textbf{compatible} if $\sigma^\ast \alpha = \overline{\alpha}$. 
\end{definition}

A $\sigma$-compatible function is real-valued on $X_\R = \mathrm{Fix}( \sigma)$, while the $\R$-flow of a $\sigma$-compatible vector field preserves the totally real submanifold $X_\R$, see Huang \cite{CaloCarleman}*{Lemma 2.2}.

 \begin{definition} \label{def-tangential-dp}
    Let $X$ be a complex manifold endowed with a holomorphic symplectic form $\omega$ and an antiholomorphic involution $\tau \colon X \to X$. We say that $(X, \omega, \sigma)$ admits the $\sigma$-\textbf{Hamiltonian density property} ($\sigma$-\textbf{symplectic density property}), if the complex Lie algebra generated by $\C$-complete $\sigma$-compatible holomorphic Hamiltonian (resp.\@ symplectic) vector fields on $X$ is dense in the Lie algebra of holomorphic Hamiltonian (resp.\@ symplectic) vector fields on $X$, in the compact-open topology. 

    We denote by $\saut_{\omega} (X)$ the group of $\sigma$-equivariant, holomorphic symplectic automorphisms of $X$:
    \begin{align*}
        \saut_{\omega} (X) = \{\, \Phi \in \aut_{\omega}(X) : \Phi \circ \sigma = \sigma \circ \Phi \,\}.
    \end{align*}
    If $(X, \omega, \sigma)$ is a symplectic complexification of $(X_\R, \omega_\R)$, then
    \[
        \saut_\omega (X) = \{\, \Phi \in \aut_{\omega}(X) : \Phi (X_\R) = X_\R \,\}.
    \]
\end{definition}

\begin{remark} \label{replace-Lie-by-Sum}
    Since $[V, W] = \lim_{t \to 0} (\varphi_t^* W - W)/t$, where $\varphi_t$ denotes the time-$t$ flow of $V$, $(\varphi_t^* W - W)/t$ for small $t$ approximates $[V,W]$ uniformly on compact subsets of $X$. 
    Moreover the holomorphic functions modulo locally constant functions $\holo(X)/\C$ with the Poisson bracket is (anti-)isomorphic to the Lie algebra of Hamiltonian vector fields on $X$. Thus the $\sigma$-Hamiltonian density property of $X$ is equivalent to the fact that the complex vector space generated by real holomorphic Hamiltonian functions inducing complete Hamiltonian vector fields is dense in $\holo(X)/\C$ in the compact-open topology. 
\end{remark}

\begin{definition}
    Let $M$ be a smooth affine algebraic variety and $\tau \colon M \to M$ an antiregular involution. 
    We say that $M$ is $\tau$-\textbf{flexible}, if $\tau$-compatible LNDs on $M$ span the tangent space $T_q M$ at every point $q \in M$. 
\end{definition}

Note that $\tau$-compatibility is preserved under automorphisms in $\taut_{alg}(M)$.
\begin{lemma} \label{tauFlex-Properties}
    Let $M$ be a smooth affine algebraic variety of dimension $\ge 2$ and $\tau \colon M \to M$ an antiregular involution. If $M$ is $\tau$-flexible, then 
    \begin{enumerate}
        \item  the group $\taut_{alg}(M)$ of equivariant algebraic automorphisms acts transitively on each connected component of $M_\R$;
        \item  there are finitely many $\tau$-compatible LNDs which span the tangent space $T_q M$ at every point $q \in M$.
    \end{enumerate}
\end{lemma}
\begin{proof}
    (1) Let $q \in M_\R$. By assumption there are $n$ $\tau$-compatible LNDs $V_1, \dots, V_n$ span the tangent space $T_q M_\R$ at $q$. By the inverse function theorem, the map
    \[
        \R^n \to M_\R, \quad (t_1, \dots, t_n) \mapsto \phi_{t_n}^n \circ \cdots \circ \phi_{t_1}^1,
    \]
    where $\phi_{t_i}^i$ is the flow map of $V_i$, is a local diffeomorphism on a neighborhood around $0 \in \R^n$. This shows that the $\taut_{alg}(M)$-orbit through $q$ is open in $M_\R$. Since $M_\R$ is foliated by $\taut_{alg}(M)$-orbits, the orbit through $q$ is also closed. The transitivity of $\taut_{alg}(M)$ on each connected component of $M_\R$ follows. 
    \smallskip

    (2)  Since linear independence of vector fields is an open condition, the set where $V_1, \dots, V_n$ do not span the tangent space is a proper algebraic subvariety $A$ of $M$. Let $q' \in A$ and $V_{n+1}, \dots, V_{2n}$ be $\tau$-compatible LNDs spanning $T_{q'} M$. Then, $V_1, \dots, V_{2n}$ span the tangent spaces outside a subvariety of $A$. An induction on the dimension of sets where the existing LNDs do not span shows the claim.
\end{proof}

\begin{proposition} \label{alg-tauFlex-to-SDP}
    Let $M$ be a smooth affine algebraic variety and $\tau \colon M \to M$ an antiregular involution. If $M$ is $\tau$-flexible, then its algebraic cotangent bundle $T^* M$ has the $\tilde{\tau}$-Hamiltonian density property and the $\tilde{\tau}$-symplectic density property. 
\end{proposition}
\begin{proof}
    The proof is an equivariant version of Section \ref{sec: KK-Method}. We walk through the same steps using similar notations and explain the necessary adjustments.  
    \smallskip

    In dimension 1, it is straightforward to verify (using Lemma \ref{lem: C[h_V]} and Lemma \ref{lem: fgIHF}) that real polynomials in $IHF_{T^* \C}$ generate $\C[T^* \C]$ as complex Lie algebra. We may assume $\dim M \ge 2$.

    As in Lemma \ref{SDP-to-HDP}, we reduce the symplectic case to a Hamiltonian one.  
    For a symplectic vector field $V$ on $\cotbun$, there is a closed 1-form $\alpha$ on $M$ representing $[i_V \omega]$ in $H^1_{DR}(M) \cong H^1_{DR}(T^*M)$. Since
    \begin{align*}
        \alpha = \frac{\alpha + \overline{\tau^* \alpha}}{2} + i \frac{\alpha - \overline{\tau^* \alpha}}{2i},
    \end{align*}
    the vertical lifts $V_1, V_2$ of the $\tau$-compatible closed 1-forms $(\alpha + \overline{\tau^* \alpha})/2$ and $(\alpha - \overline{\tau^* \alpha})/(2i)$ respectively, are complete $\ttau$-compatible symplectic vector fields. Thus, the difference $V - V_1 - i V_2$ is Hamiltonian.
    \smallskip

    Since $\ttau^* \omega = \overline{\omega}$ by Lemma \ref{T^*M-SympCompl}, $\ttau$-compatible functions on $T^*M$ induce $\ttau$-compatible Hamiltonian vector fields. 
    
    By Lemma \ref{tauFlex-Properties} (3), there are $\tau$-compatible LNDs $V_1, V_2, \dots, V_N$ spanning the tangent bundle $TM$. For the algebraic Hamiltonian density property of $X=T^*M$, we need to show that  
    \begin{align}  \label{HDP-tau}
        O(M) [h_{V_1},h_{V_2},\dots, h_{V_N}] \subset \lie_X^{\ttau},
    \end{align}
    where $\lie_X^{\ttau}$ denotes the complex Lie algebra generated by $\ttau$-compatible functions in $IHF_X$ whose Hamiltonian vector fields are complete. 
    \smallskip

    Since $\tau$-compatible functions induce $\tau$-compatible 1-forms, which in turn has complete $\ttau$-compatible vertical lift, we have $O^\tau(M) \subset \lie_X^{\ttau}$. Moreover, every function $f \in O(M)$ can be written as 
    \begin{align*}
        f = \frac{f + \overline{\tau^* f}}{2} + i \frac{f - \overline{\tau^* f}}{2i}, 
    \end{align*}
    where the first summand is $\tau$-compatible and the second summand is $i$ times a $\tau$-compatible function. 
    Hence, $O(M) = O^\tau(M) + i \, O^\tau(M) \subset \lie_X^{\ttau}$.
    \smallskip

    For a $\tau$-compatible LND $V$ on $M$, we have $\C[M][h_V^k] \subset \lie_X^{\ttau}$ for any $k \in \N$ by Lemma \ref{mix-field-function}.

    Next, the proof of Lemma \ref{Stalk-at-pt} constructs out of 
    \[
        \{ \, f h_{V}^j: \, f \in \C[M],  \, V \in \vect_{\mathrm{alg}}(M) \,\, \tau\text{-compatible LND}, j \le k \,\}
    \]
    by iterative application of Lie derivatives a coherent $O_M$-submodule $\mathcal{E}_k^{\ttau}$ with the property $(\mathcal{E}_k^{\ttau})_q = (\mathcal{F}_k)_q$. 
    \smallskip

    Now, start with finitely many $\tau$-compatible sections $s_i$ spanning $\oplus_{j \le k} \sym^j(T_qM)$ at a point $q$. The set of points where they do not span is a subvariety $B$ of lower dimension in $M$. Let $q' \in B$. By the $\tau$-flexibility of $M$, there are $\tau$-compatible LNDs $W_1, \dots, W_n$ spanning the tangent space at $q'$. The real flow maps of these $n$ $\tau$-compatible LNDs form local coordinates for a totally real submanifold of maximal dimension around $q'$. Since the $n$-dimensional real span of $W_{1}, \dots, W_{n}$ cannot be contained in $T_{q'} B$, there is an equivariant automorphism $\alpha \in \taut_{alg}(M)$ such that $\alpha(q') \notin B$. Then by adding $\alpha^* s_i$ to $\{ s_i \}$, we reduce the dimension of the subvariety where the sections do not span. Since $\alpha$ is in $\taut_{alg}(M)$, the generating set of $\mathcal{E}_k^{\tilde{\tau}}$ is stable under pullback of $\alpha$. Proceed as in Lemma \ref{stalks-to-sheaf} to obtain a finite set of sections spanning $\oplus_{j \le k} \sym^j TM$. 
    \smallskip

    Finally, a standard application of Theorem B implies that the global sections of $\mathcal{E}_k^{\ttau}$, $\mathcal{F}_k$ are identical and thus all global sections of $\mathcal{F}_k$ are in $\lie_X^{\ttau}$. Since this holds for arbitrary $k \in \N$, we obtain \eqref{HDP-tau} and the $\ttau$-Hamiltonian density property of $\cotbun$.
\end{proof}

\begin{definition} \label{tau-hol-flex}
    Let $M$ be a Stein manifold and $\tau \colon M \to M$ an antiholomorphic involution. 
    We say that $M$ is $\tau$-\textbf{holomorphically} \textbf{flexible}, if $\tau$-compatible complete holomorphic vector fields on $M$ span the tangent space $T_q M$ at every point $q \in M$.
\end{definition}

\begin{lemma} \label{tauHolFlex-Properties}
    Let $M$ be a Stein manifold of dimension $\ge 2$ and $\tau \colon M \to M$ an antiholomorphic involution. If $M$ is $\tau$-holomorphically flexible, then 
    \begin{enumerate}
        \item  the group $\taut(M)$ of equivariant holomorphic automorphisms acts transitively on each connected component of $M_\R$;
        \item  there are finitely many $\tau$-compatible complete holomorphic vector fields which span the tangent space $T_q M$ at every point $q \in M$.
    \end{enumerate}
\end{lemma}
\begin{proof}
    The proof is similar to the affine case, Lemma \ref{tauFlex-Properties}. For (2), the analytic subset $B$ of $M$ where the existing finitely many $\tau$-compatible vector fields $V_1, V_2, \dots, V_N$ do not span can have countably many irreducible components $B_1, B_2, \dots$. There exists by Kutzschebauch \cite{MR3229363}*{Lemma 25} an equivariant automorphism $\phi$ such that $\phi(M \setminus B) \cap B_i \neq \emptyset$ for every $i = 1, 2, \dots$. Adding $\phi_* V_i$'s to the $V_i$'s, we reduce the dimension of the singular variety and may proceed with the induction.
\end{proof}

\begin{proposition} \label{holo-tauFlex-to-SDP}
    Let $M$ be a Stein manifold and $\tau \colon M \to M$ an antiholomorphic involution. If $M$ is a $\tau$-holomorphically,
    then its holomorphic cotangent bundle $T^* M$ has the $\tilde{\tau}$-Hamiltonian density property and the $\tilde{\tau}$-symplectic density property. 
\end{proposition}
\begin{proof}
    Replace ``algebraic" in the proof of Proposition \ref{alg-tauFlex-to-SDP} by ``holomorphic".
\end{proof}

\medskip 

Let $(X, \omega)$ be a smooth symplectic manifold. A {\bf symplectic isotopy} is a jointly smooth map $\phi \colon [0, 1] \times X \to X$ such that $\phi_t$ is symplectic for every $t$ in $[0,1]$ and $\phi_0 = \id$. The isotopy $\phi_t$ is generated by a smooth family of vector fields $V_t$ with $d \phi_t / dt= V_t \circ \phi_t, \, \phi_0 = \id$. Since $\phi_t$ is symplectic, the vector field $V_t$ is symplectic and the 1-form $i_{V_t} \omega$ is closed for $t \in [0,1]$. A {\bf Hamiltonian isotopy} is a symplectic isotopy such that the closed 1-form $i_{V_t} \omega$ is exact, namely $V_t$ is Hamiltonian for all $t \in [0,1]$. 
We call a symplectomorphism which is the endpoint of a Hamiltonian isotopy a \textbf{Hamiltonian diffeomorphism}. Hamiltonian diffeomorphisms form a normal subgroup of the group of symplectic diffeomorphisms, see Dusa McDuff and Dietmar Salamon \cite{MR3674984}*{\S 3.1}. 
\smallskip

The following is a general complexification theorem for Hamiltonian diffeomorphisms, from  which we will deduce the  main result of this paper, Theorem \ref{thm-Carleman} from the Introduction.
\begin{theorem} \label{Ham-Carleman}
    Let $(X_\R, \omega_\R)$ be a smooth symplectic manifold which admits a symplectic complexification $(X, \omega, \sigma)$. If $X$ is Stein and has the $\sigma$-Hamiltonian density property, then for every Hamiltonian diffeomorphism $\varphi \in \diff_{\omega_\R}(X_\R)$, for any $k \in \N$ and any positive continuous function $\epsilon$ on $X_\R$, there exists a holomorphic symplectic automorphism $\Phi \in \saut_{\omega}(X)$ such that
	\[	d_k( j^k \Phi (p) - j^k \varphi(p) ) < \epsilon (p) \quad \text{ for all }\, p \in X_\R.	\]
\end{theorem}

The proof of Theorem \ref{Ham-Carleman} goes by induction, where at each induction step we use the $\sigma$-Hamiltonian density property. We need an appropriate exhaustion of $X$ by $\sigma$-invariant, holomorphically convex compact subsets $S_j, j=1, 2,\dots$. The forthcoming considerations will be used to construct these subsets. 
\begin{lemma} \label{lem: Stein-in-CN}
    Every Stein manifold $X$, endowed with an antiholomorphic involution $\sigma \colon X \to X$, admits a holomorphic embedding $G \colon X \to \C^q$ for some $q \in \N$ such that $G(X_\R) = G(X) \cap \R^{q}$. 
\end{lemma}
\begin{proof}
    Since $X$ is Stein, there are $N \in \N$ and  holomorphic functions $f_i \colon X \to \C$ such that
\[
    F \colon X \hookrightarrow \C^N, x \mapsto (f_1(x), f_2(x), \dots, f_N(x))
\]
is a proper holomorphic embedding. 
Then, we can embed $X$ into $\C^{2N}$ as
\begin{align*}
    G \colon X \hookrightarrow \C^{2N}, x \mapsto (g_1(x), g_2(x), \dots, g_{2N}(x)),
\end{align*}
where
\begin{align*}
    g_{2k-1} = f_k + \overline{f_k \circ \sigma} \, \text{ and }  \, g_{2k} = i(f_k - \overline{f_k \circ \sigma} ), \, k = 1, 2, \dots, N, \, \, \text{ are real holomorphic.} 
\end{align*}
Since $G (\sigma(p)) = \overline{G(p)}$ for $p \in X$, the fixed-point set $X_\R$ is being embedded into $\R^{2N} \subseteq \C^{2N}$.  Moreover, a simple calculation shows that $G(X_\R) = G(X) \cap \R^{2N}$.
\end{proof}

From now on, we view $X$ as a subvariety of $\C^{q}$ and $\sigma$ as the restriction of complex conjugation of $\C^{q}$ to $X$.  

For $p_0$ in $X_\R$, the function
$\C^{q} \to \R_{\ge 0}, \, p \mapsto \| p - p_0 \|^2$ is $\sigma$-invariant and thus its restriction $\rho \colon X \to \R_{\ge 0}$ to $X$  is a $\sigma$-invariant strictly plurisubharmonic exhaustion function on $X$. 
We consider the closed sublevel set 
	\begin{align*} 
		Z_r = \{ p \in X: \rho(p) \le r \} \quad \text{ and } \quad	Z_r^\R = Z_r \cap X_\R	
	\end{align*}
the intersection of $Z_r$ with the fixed-point set. 

\begin{lemma} \label{Saturn-holoconv}
    The union $S:= Z_{r_1}^\R \cup Z_{r_2}$ is holomorphically convex in $X$ for $r_1, r_2 \in  \R_{\ge 0}$. 
\end{lemma}
\begin{proof}
    The only nontrivial case is $r_1 > r_2 > 0$. 
    By a result of Evgenii Chirka and Maxim Smirnov \cite{MR1155560}*{Theorem 2}, the union 
    \[
        \{ p \in \R^{q}: \| p - p_0 \|^2 \le r_1 \} \cup \{ p \in \C^{q}: \| p - p_0 \|^2 \le r_2 \}
    \]
    is polynomially convex in $\C^{q}$. Since $X$ is a complex subvariety of $\C^{q}$, its intersection $ Z_{r_1}^\R \cup Z_{r_2}$ with $X$ is holomorphically convex in $X$ by Cartan's Theorem B. 
\end{proof}

Here comes the local version of the complexification result Theorem \ref{Ham-Carleman},
an approximation of Anders{\'e}n-Lempert type for symplectic diffeomorphisms of $X_\R$ uniformly on the compact subsets $S$. 
\begin{theorem} \label{theorem: LocalCarleman}
	Let $(X_\R, \omega_\R)$ be a smooth symplectic manifold which admits a symplectic complexification $(X, \omega, \sigma)$. Let $\varphi \colon [0,1] \times X_\R \to X_\R$ be a Hamiltonian isotopy and $r \ge 0$ such that $\varphi_t \rvert_{Z_r^\R} = \id_{Z_r^\R}$ for all $t$ in $[0,1]$. If $X$ is Stein and has the $\sigma$-Hamiltonian density property, then for any $k \in \N$, any $\epsilon>0$, any $b \in (0,r)$ and any $a > b$, there exists a holomorphic Hamiltonian isotopy $\Phi \colon [0,1] \times X \to X$ such that $\Phi_{t} \in \saut_{\omega}(X)$ for all $t$ in $[0,1]$ and 
	\begin{align*}
		  d_k( j^k \Phi_{t}(p) - j^k \varphi_t(p)) &< \epsilon \text{ for all } p \in Z_{a}^\R, 	\\
		d_k( j^k \Phi_{t} (p) - j^k \mathrm{id} (p) )		&< \epsilon  \text{ for all } p \in Z_{b}.	 
	\end{align*}
\end{theorem}
\begin{proof}
    We introduce an auxiliary time variable $s$
    \[
        \psi \colon [0,1] \times [0,1] \times X_\R \to X_\R, (s, t, p) \mapsto \varphi(st, p)
    \]
    and consider $\varphi_t$ as the time-1 Hamiltonian flow of the time-$s$ dependent Hamiltonian function $h(s,t,p)$ 
    with the initial condition $\psi(0,t,p) = \varphi(0,p) = p$. 
    Since $\varphi_t \rvert_{Z_r^\R} = \id_{Z_r^\R}$ for $t \in [0,1]$, we can assume that $h(s,t,\cdot)$ is $0$ on $Z_r^\R$ and thus we can extend $h_{s,t}$ to be 0 in an open neighborhood of $Z_{b'}$ for $s, t \in [0,1]$, where $r > b' > b$. 
    Choose $a'>0$ such that 
    \begin{align*}
        \bigcup_{ t \in [0,1] }  \varphi_t ( Z_{a}^\R ) \subset Z_{a'}^\R. 
    \end{align*}
    Let $Q$ be a compact Hausdorff space whose points we denote by $q$. Given a continuously $q$-dependent function $h_q$ on $S= Z^\R_{a'} \cup Z_{b'}$ which is in $C^{k+1}(S) \cap \holo(Z_{b'})$, by the holomorphic convexity of $S$ in $X$ (Lemma \ref{Saturn-holoconv}) and the parametric version of the Mergelyan theorem by Forn{\ae}ss--Forstneri{\v c}--Wold \cite{MR4264040}*{Theorem 20}, $h_q$ is approximable in the $C^{k+1}$-norm by a continuously $q$-dependent function $P_q$ which is holomorphic in a neighborhood $\Omega$ of $S$. 
    Hence $h_{s,t}$ can be approximated in the $C^{k+1}$-norm on $S$ by a continuously $(s,t)$-dependent holomorphic function $P_{s,t}$ for $s,t \in [0,1]$.
    Moreover, we can assume that $P(s,t,p)$ is a finite sum of the form $\sum_{n,m} s^n t^m P_{nm}(p)$ by a patching argument as in the proof of \cite{MR1314745}*{Lemma 1.2}. 

    Then by the Cartan--Oka--Weil theorem, we can approximate $P_{nm}$ on a relative compact Runge open neighborhood of $S$ by a holomorphic function $\widetilde{P}_{nm}$ on $X$. Since both $P_{nm}$ and $\widetilde{P}_{nm}$ are  holomorphic in a neighborhood of $S$, the approximation holds in the $C^{k+1}$-norm.

    Since $h_{s,t}$ is a real function on $X_\R$, we can replace $\widetilde{P}_{nm}$ by $(\widetilde{P}_{nm} + \overline{\sigma_* \widetilde{P}_{nm}})/2$ such that 
    \[
        \widetilde{P}_{s,t}(p) = \widetilde{P}(s,t,p) = \sum_{n,m} s^n t^m \widetilde{P}_{nm}(p)
    \]
    is real holomorphic with respect to $\sigma$ and approximates $h_{s,t}$. 

    By the $\sigma$-density property and Remark \ref{replace-Lie-by-Sum}, $\widetilde{P}_{nm}$ can be approximated on $\Omega$ by a complex linear combination $L$ of real holomorphic functions in $IHF_X$. 
    Since real holomorphic functions are real-valued on $X_\R$, we may assume that the coefficients in this linear combination are real. Indeed we can separate the coefficients of $L$ and thus also $L$ into real and imaginary parts $\widetilde{P}_{nm} \approx L_1 + i L_2$. Acting on both sides with $\overline{\sigma_*(\cdot)}$, we get $\widetilde{P}_{nm} \approx L_1 - i L_2$ since $\widetilde{P}_{nm}$ and $L_1, L_2$ are real holomorphic. Thus $L_2 \rvert_{\Omega} \approx 0$ and $L_1$ approximates $\widetilde{P}_{nm}$ on $\Omega$. Hence $\widetilde{P}_{s,t}$ can be approximated on $\Omega$ (thus $h_{s,t}$ on $S$) by a finite sum $\sum_{l=1}^N p_l(s,t) \widetilde{P}_l$ where $p_l(s,t)$ is a polynomial and $\widetilde{P}_l$ is a real holomorphic function in $IHF_X$.

    By definition for all $t \in [0,1]$ the Hamiltonian vector field corresponding to $h_{s,t}$ is integrable for time $s \in [0,1]$ for all starting points $x \in Z_a^\R \cup Z_b$ and its flow is contained in $S = Z^\R_{a'} \cup Z_{b'} \subset \Omega$, thus there is an open neighborhood $U$ containing $Z_a^\R \cup Z_b$ such that the Hamiltonian vector field $\widetilde{W}$ corresponding to $\sum_{l=1}^N p_l(s,t) \widetilde{P}_l$ is integrable for time $s \in [0,1]$ for all initial points $x$ in $U$. 

    For a fixed $n \in \N$, $j \in \{0, 1, \dots, n-1\}$, we denote by $\Phi^{(s, l)}_\theta$ the time-$\theta$ Hamiltonian flow map of the time-$s$ independent function $ p_l(s,t) \widetilde{P}_l$, $l \in \{ 1, \dots, N\}$ and consider the composition
    \begin{align*} 
        \Phi^{(s)}_{\theta} =  \Phi^{(s, 1)}_{\theta} \circ \Phi^{(s, 2)}_{\theta} \circ \cdots \circ \Phi^{(s,N)}_{\theta} .
    \end{align*}
    Now we take 
    \begin{align} \label{composition-ham}
		 \Phi^{(\frac{n-1}{n})}_{1/n} \circ  \Phi^{(\frac{n-2}{n})}_{1/n}  \circ \cdots \circ  \Phi^{(\frac{1}{n})}_{1/n} \circ  \Phi^{(0)}_{1/n} 
	\end{align}
    which approximates the time-$1$ flow of $\widetilde{W}$ uniformly on a relatively compact open neighborhood of any compact subset in $U$ for large enough $n$ \cite{MR0515141}*{Theorem 2.1.26}. Indeed an algorithm for the vector field $\left( 1,  \widetilde{W} \right) $ on $[0,1] \times ([0,1]\times X)$ is given by 
    \begin{align*}
        [0, \epsilon) \times [0,1]^2 \times X \to  [0,1]^2 \times X, \quad (\theta, s, t, p) \mapsto (s + \theta, t, \Phi^{(s)}_\theta(p))  
    \end{align*}
    whose $n$th iterate after projecting to $X$ is \eqref{composition-ham}. 
    By the Cauchy estimate, the approximation is in the $C^k$-norm and \eqref{composition-ham} depends smoothly on $t$. 
    
    Finally, since approximation of $h_{s,t}$ by the function $\sum_{l=1}^N p_l(s,t) \widetilde{P}_l$ implies the approximation of $\varphi_t$ by the time-1 flow of $\widetilde{W}$, it follows that \eqref{composition-ham} approximates $\varphi_t$ on a relatively compact open neighborhood of $Z^\R_a \cap Z_b$. 
\end{proof}

\noindent \textit{Proof of} Theorem \ref{Ham-Carleman}. 
Since $\varphi$ is a Hamiltonian diffeomorphism of $X_\R$, there exists a Hamiltonian isotopy $\varphi_t$ in $\diff_{\omega_\R}(X_\R), t \in [0,1]$ with $\varphi_0 = \mathrm{id}$ and $\varphi_1= \varphi$. 
	Given any natural number $j$ and any positive $\varepsilon_j$, assume that we have a holomorphic symplectomorphism $\Phi_j \in \saut_\omega(X)$, a Hamiltonian isotopy $\psi_{j,t}$ of $X_\R$, real numbers $r_j, s_j$ with $r_j \ge j-1, s_j \ge r_j + 1$ such that
	\smallskip
	\begin{enumerate}
		\item[($1_j$)] 	The image of $ Z_{r_j} $ under $\Phi_j$ is contained in $Z_{s_j}$. 
		\item[($2_j$)]	For $j \ge 2$ 
					\[	 d_k (j^k \Phi_j (p) - j^k \Phi_{j-1}(p) ) < \varepsilon_j \quad  \forall \, p \in Z_{r_{j-1}}.
                    \] 
		\item[($3_j$)]	On $Z^\R_{s_j }$ we have $\psi_{j,t}$ is the identity for $t$ in $[0,1]$. 
		\item[($4_j$)]	For all $p$ in $X_\R$
					\[	d_k (j^k (\psi_{j, 1} \circ \Phi_j) (p) - j^k \varphi (p) ) < \epsilon (p).
                    \]
	\end{enumerate}

    This inductive process builds on the uniform approximation on the compact subsets $S_j$ of the form $Z_a^\R \cup Z_b, a > b$ as in Theorem \ref{theorem: LocalCarleman} and the so-called pushed-out method which enlarges the approximating region at each inductive step. For more details see Huang \cite{CaloCarleman}*{Theorem 6.1}. 
    \smallskip

    The desired $\Phi \in \saut_\omega(X)$ that should approximate $\varphi$ on $X_\R$ is the limit $\lim_{j \to \infty} \Phi_j$. Indeed the holomorphic symplectomorphism $\Phi_{j+1} \circ \Phi_j^{-1}$ approximates the identity on $Z_{r_j}$. Since $r_j \ge j -1$, the sublevel set $Z_{r_j}$ exhausts $X$ in the limit. Additionally we have $r_{j+1} = s_j + 1 > r_j + 2$. Choose $\{ \varepsilon_j \}_{j \in \N} \subset \R_{>0}$ to have finite sum such that
	\[	0 < \varepsilon_{j+1} < \mathrm{dist} ( Z_{r_{j}} , X \setminus Z_{r_{j+1}} ),
	\]
	where $\mathrm{dist}(\cdot, \cdot)$ is the distance function on $\C^{2N}$ restricted to $X$. 
	Therefore the limit 
	\[	\lim_{j \to \infty} \Phi_{j+1} \circ \Phi_j^{-1}
	\]
	exists uniformly on compact subsets on 
	\[	
        \bigcup_{j = 1}^{\infty} \left( \Phi_{j+1} \circ \Phi_j^{-1} \right)^{-1}  (Z_{r_j}) = X
	\]
	(see e.g.\@ Forstneri{\v c} \cite{MR1760722}*{Proposition 5.1}). Thus
	\[	
         \lim_{j \to \infty} \Phi_{j} = (\lim_{j \to \infty} \Phi_{j+1} \circ \Phi_j^{-1} ) \circ \Phi_1
	\]
	converges to a holomorphic symplectomorphism $\Phi$ in $\saut_\omega(X)$. 
	Moreover, $(1_j)$ and $(3_j)$ guarantee that $\psi_{j, 1} \circ \Phi_j$ is $\Phi_j$ on $Z_{r_j}^\R$. Then $(4_j)$ says $\Phi_j$ approximates $\varphi$ on $Z_{r_j}^\R$ and thus in the limit $\Phi$ approximates $\varphi$. 
    \hfill \qed

\begin{theorem} \label{thm: Carleman-Cotan}
    Let $M$ be a Stein manifold and $\tau \colon M \to M$ an antiholomorphic involution such that the fixed-point set $M_\R=\mathrm{Fix}(\tau) \neq \emptyset$. Assume that $M$ is $\tau$-holomorphically flexible. 
    
    Then for every Hamiltonian diffeomorphism $\varphi \in \diff_{\omega_\R}(T^*_\R M_\R)$, for any $k \in \N$ and any positive continuous function $\epsilon$ on $T^*_\R M_\R$, there exists a holomorphic symplectomorphism $\Phi \in \saut_{\omega}(\cotbun)$ such that
	\[	d_k( j^k \Phi (p) - j^k \varphi(p) ) < \epsilon (p) \quad \text{ for all } p \in \rcotbun.	\]
\end{theorem}
\begin{proof}
    Consider the symplectic complexification $(\cotbun, \omega, \ttau)$ of $(\rcotbun, \omega_\R)$. The holomorphic cotangent bundle $\cotbun$ over a Stein base $M$ is Stein and has the $\tilde{\tau}$-Hamiltonian density property by Proposition \ref{holo-tauFlex-to-SDP}.  Thus Theorem \ref{Ham-Carleman} gives the desired approximation. 
\end{proof}

\section{Real Flexible Examples} \label{sec: realFlex}
In this section we present examples where our approximation results apply. 

\subsection{Affine examples}

\begin{lemma} \label{sl2-triple}
    Let $\mathfrak{g}$ be a Lie algebra and $E,F \in \mathfrak{g}$ two nilpotent elements generating an $\mathfrak{sl}_2$-subalgebra, namely $[[E,F],E]=2E$ and $[[E,F],F]=-2F$. Then $E-F-[E,F]$ is nilpotent. 
\end{lemma}
\begin{proof}
    Let $X = E-F-[E,F]$ and $\mathfrak{s} = \langle E, F, [E,F] \rangle$. The matrix of the endomorphism $ad_X \colon \mathfrak{s} \to \mathfrak{s}, Y \mapsto [X,Y]$ with respect to the basis $\{ E,F,[E,F]\}$ 
    \begin{align*}
        \begin{pmatrix}
            -2 & 0 & -2 \\ 0 & 2 & -2 \\ 1 & 1 & 0
        \end{pmatrix} 
    \end{align*} 
    is nilpotent.
\end{proof}

A real form $G_\R$ of a linear algebraic group $G$ is called \textbf{noncompact} if the semisimple part of the Lie algebra $\mathfrak{g}_\R$ of $G_\R$ does not contain any compact Lie ideal.
\begin{proposition}
    Let $G$ be a linear algebraic group whose identity component has no nontrivial characters and $G_\R$ a noncompact real form of $G$ with the antiregular involution $\tau$. 
    
    (1) Then the group $G$ is $\tau$-flexible. 
    
    (2) Moreover let $H$ be a complex closed subgroup of $G$ such that $G/H$ is affine. Then the homogeneous space $G/H$ is $\tau$-flexible. 
\end{proposition}
\begin{proof}
    (1) Since all components of $G$ are isomorphic as varieties, we may assume $G$ is connected. By \cite{MR2768646}*{Lemma 1.1} the radical of $G$ coincides with the unipotent radical. By the Levi decomposition it suffices to consider the semisimple case.

    Consider the Cartan decomposition of the Lie algebra of $G_\R$, $\mathfrak{g}_\R = \mathfrak{k}_0 \oplus \mathfrak{p}_0$,
    where $\mathfrak{k}_0$ and $\mathfrak{p}_0$ are the $(+1)$-eigenspace and $(-1)$-eigenspace of the Cartan involution $\theta$.
    Let $\mathfrak{a}_0$ be a maximal abelian subalgebra of $\mathfrak{p}_0$ and $Z_{\mathfrak{k}_0}(\mathfrak{a}_0)$ the centralizer of $\mathfrak{a}_0$ in $\mathfrak{k}_0$. Since $G_\R$ is noncompact, the maximal split subtorus $\mathfrak{a}_0 \neq \{0\}$. 
    For $\lambda \in \mathfrak{a}_0^*$, let
    \begin{align*}
        \mathfrak{g}_\lambda = \{\, X \in \mathfrak{g}_\R : (\mathrm{ad}h)X = \lambda(h) X \text{ for all } h \in \mathfrak{a}_0 \,\}.
    \end{align*}
    If $\mathfrak{g}_\lambda \neq \{0\}$ and $\lambda \neq 0$, we say that $\lambda$ is a restricted root of $\mathfrak{g}_\R$ and denote by $\Sigma$ the set of restricted roots. 
    The restricted root space decomposition for $\mathfrak{g}_\R$ is given by (see Anthony Knapp \cite{MR1920389}*{Proposition 6.40})
    \begin{align*}
        \mathfrak{g}_\R = \mathfrak{g}_0 \oplus \bigoplus_{\lambda \in \Sigma} \mathfrak{g}_\lambda \text{ with } \mathfrak{g}_0 = \mathfrak{a}_0 \oplus Z_{\mathfrak{k}_0}(\mathfrak{a}_0).
    \end{align*}
    The nonzero restricted root spaces $\mathfrak{g}_\lambda$, $\lambda \in \Sigma$ are nilpotent. 

    Denote by $\mathfrak{l} = \langle \mathfrak{g}_\lambda, \lambda \in \Sigma \rangle$ the Lie algebra generated by the nonzero restricted root spaces. Then $\mathfrak{l}$ is a Lie ideal of $\mathfrak{g}_\R$ by $[\mathfrak{g}_0, \mathfrak{g}_\lambda] \subset \mathfrak{g}_\lambda$.
    For each $X \in \mathfrak{g}_\lambda \setminus \{0\}$, we have $\theta(X) \in \mathfrak{g}_{-\lambda}$ and $[X, \theta(X)] \in \mathfrak{a}_0$. Moreover, the triple $(X, \theta(X), [X, \theta(X)])$ form an $\mathfrak{sl}_2$-triple. By Lemma \ref{sl2-triple}, $\mathfrak{a}_0 \subset \mathfrak{l}$. Since the maximal abelian subalgebra $\mathfrak{a}_0 \subset \mathfrak{p}_0$ intersects every noncompact simple Lie ideal of $\mathfrak{g}_\R$ nontrivially, it follows that $\mathfrak{l}$ is the sum of all noncompact simple Lie ideals of $\mathfrak{g}_\R$, i.e.\@ $\mathfrak{l}=\mathfrak{g}_\R$. Hence every vector in $\mathfrak{g}_0 = \mathfrak{a}_0 \oplus Z_{\mathfrak{k}_0}(\mathfrak{a}_0)$ can be written as a Lie combination of nonzero restricted root vectors. 

    Let $h \in Z_{\mathfrak{k}_0}(\mathfrak{a}_0)$. Our aim is find a nilpotent vector in $\mathfrak{g}_\R$ whose projection to $\R h \subset Z_{\mathfrak{k}_0}(\mathfrak{a}_0)$ is nonzero. For this purpose, it suffices to consider the case when $h$ is a nested Lie bracket of nilpotent vectors $X_1, X_2, \dots, X_k$. Indeed, the nilpotent vector 
    \begin{align*}
        \mathrm{Ad} (\exp (t_1 X_1)) \mathrm{Ad} (\exp (t_2 X_2)) \cdots \mathrm{Ad} (\exp (t_{k-1} X_{k-1})) X_k, \quad t_i \in  \Q_{>0},
    \end{align*}
    has $h$ as the coefficient for $t_1 t_2 \cdots t_{k-1}$ in the expansion. 

    Therefore, we can find a real basis of $\mathfrak{g}_\R$ consisting of nilpotent vectors. The left-invariant LNDs induced by their complexifications in $\mathfrak{g}$
    span the tangent space at every point of the homogeneous space $G$.  
    Thus the complexification $G$ is $\tau$-flexible. 

    \smallskip
    (2) The fixed-point set $H_\R = H^\tau$ is a real form of $H$ and $G_\R / H_\R \subset (G/H)^\tau$. The same LNDs on $G$ induce $\tau$-compatible LNDs on $G/H$ and thus it is $\tau$-flexible, cf.\@ the proof in the more general setting of Proposition \ref{transitivity}. 
\end{proof}

\subsection{Stein examples}

We remind the reader about the universal complexification of a real Lie group $G_\R$, see Gerhard Hochschild \cite{MR0206141}. A complex Lie group $G$ is called the {\bf universal complexification} of $G_\R$, if there exists a continuous (and thus automatically real analytic) group homomorphism $i \colon G_\R \to G$, so that the following universal property holds:

\smallskip
For any real analytic group homomorphism $\varphi \colon G_\R \to H$ in a complex Lie group $H$ there exists a unique holomorphic group homomorphism
$\varphi^\C \colon G \to H$ such that the following diagram commutes:
    \[
    \begin{tikzcd}
    {} & G \arrow[dr, dashed]{}{\exists ! \, \varphi^\C} \\
    G_\R \arrow{ur}{i} \arrow{rr}{\varphi} && H
    \end{tikzcd}
    \]
It was proven by Hochschild that the universal complexification always exists and that there is an antiholomorphic involution $\tau_G \colon G \to G$
(which is a group homomorphism) whose fixed-point set contains the image
$i(G_\R)$. The fixed-point set of the involution $\tau_G$ is in general bigger than $G_\R$; an example is the inclusion of the group of real matrices with positive determinant $GL_n^+ (\R)$ into $GL_n (\C)$. The involution $\tau_G$ is complex conjugation and its fixed-point set is
$GL_n (\R)$.

In general $i$ is not injective, e.g.\@ the universal covering group $\widetilde{ SL_2 (\R)}$ of $SL_2(\R)$ has the universal complexification $SL_2 (\C)$ where $i = {\rm incl} \circ \pi$ is the composition of the universal covering map  $\pi \colon \widetilde{ SL_2 (\R)} \to  SL_2 (\R)$ and the inclusion of $SL_2 (\R)$ into $SL_2 (\C)$. There are also examples where the kernel of $i$ has positive dimension. 

We now give a big class of  examples where our main complexification result, Theorem \ref{Ham-Carleman}, applies.

\begin{lemma} \label{lem: LAdecomposition}
    Let $G_\R$ be a real Lie group and $G$ its universal complexification. Assume that $i \colon G_\R \to G$ is injective and $G$ is linear algebraic. Then there exists a basis $v_1, v_2, \dots, v_n$ of the real Lie algebra $\mathfrak{g}_\R$ of $G_\R$ such that $v_i$ is either nilpotent or the 1-parameter subgroup $\{\, \exp(t v_i): t \in \R \,\}$ is isomorphic to $S^1$ or to $\R_{>0}$.
\end{lemma}
\begin{proof}
    Since $i \colon G_\R \to G$ is injective and $G$ is a complex linear algebraic group, the fixed-point set of the antiholomorphic involution is a real linear algebraic group $G(\R)$. By Hochschild \cite{MR0206141}*{Theorem 3}, $G_\R \cong i(G_\R)$ is a union of connected components (including the identity component) of $G(\R)$, hence an open Lie subgroup. 

    By the Levi decomposition, $\mathfrak{g}_\R$ is the direct sum of its unipotent radical and a reductive Lie subalgebra. Since this radical consists of nilpotent elements, we may assume $\mathfrak{g}_\R$ is reductive.   

    Consider the Iwasawa decomposition $\mathfrak{g}_\R = \mathfrak{k}_0 \oplus \mathfrak{a}_0 \oplus \mathfrak{n}_0$, where $\mathfrak{k}_0$ is the Lie algebra of the maximal compact subgroup of $G_\R$, $\mathfrak{a}_0$ is a maximal abelian subalgebra of the noncompact part of the Cartan decomposition, and $\mathfrak{n}_0$ is nilpotent. 
        
    In the compact part $\mathfrak{k}_0$, let $v$ be a basis vector of $\mathfrak{k}_0$. The closure of the 1-parameter subgroup $L_v$ generated by $v$ is a connected abelian subgroup of the Lie group of $\mathfrak{k}_0$, $\overline{L_v} = (S^1)^m$, since every compact connected abelian Lie group is a torus, see Knapp \cite{MR1920389}*{Corollary 1.103}. We may vary $v$  slightly to $w$ so that $\overline{L_{w}} \cong S^1 \subset (S^1)^m$ and $w$ remains as a basis vector of $\mathfrak{k}_0$.

    The noncompact abelian part $\mathfrak{a}$ is a split torus over $\R$ and has a basis whose elements integrate to $\R_{>0}$. 
    The nilradical $\mathfrak{n}$ of $\mathfrak{g}_\R$ consists of nilpotent elements.  
\end{proof}

\begin{proposition} \label{transitivity}
    Let $G_\R$ be a real Lie group, $H_\R$ a closed subgroup and $G$, $H$  the universal complexification of $G_\R$, $H_\R$ respectively. Assume that $i \colon G_\R \to G$ is injective, $G$ linear algebraic and $G/H$ affine of complex dimension $\ge 2$. Denote by $\tau \colon G/H \to G/H$ the antiholomorphic involution induced by $\tau_G$ on $G/H$. 
    
    Then $G/H$ is $\tau$-holomorphically flexible. 
   Thus the holomorphic cotangent bundle $T^* (G/H)$ has the $\tilde{\tau}$-Hamiltonian density property and the $\tilde{\tau}$-symplectic density property.
\end{proposition}

\begin{proof}
    By the universality $H$ is a closed subgroup of $G$ such that $\tau_H = \tau_G \rvert_H$. The complex-affine homogeneous space $G/H$ is a complexification of the real homogeneous space $G_\R/H_\R \subset \mathrm{Fix}(\tau)$. The algebraic action of $G$ on $G$ from the left commutes with the right action of $H$ and induces left-invariant vector fields on $G/H$. For an element $v$ in the Lie algebra of $G$, 
    \begin{align*}
        \tau (\exp (tv) \cdot x ) = \tau_G (\exp (t v) ) \cdot \tau(x) \quad \text{ for } x \in G/H,
    \end{align*}
    then vector fields on $G/H$ induced by the action of $G_\R \subset G$ are $\tau$-compatible. 
    By Lemma \ref{lem: LAdecomposition}, there exists a basis $v_1, v_2, \dots, v_n$ of the real Lie algebra $\mathfrak{g}_\R$ of $G_\R$ such that $v_i$ is either nilpotent or the 1-parameter subgroup $\{\, \exp(t v_i): t \in \R \,\}$ is isomorphic to $S^1$ or to
    $\R_{>0}$. Since $\tau_H = \tau_G \rvert_H$ there exist a subset of this basis whose images under $T_e G_\R \to T_{e H_\R} (G_\R / H_\R)$ form a basis $w_1, \dots, w_{m}$ of $T_{e H_\R} (G_\R / H_\R)$. 
    Then their corresponding complex left-invariant vector fields on $G/H$ span the tangent space at every point of $G/H$.
\end{proof}

\begin{theorem} \label{thm: homoSp-Ham}
Let $G_\R$ be a real Lie group, $H_\R$ a closed subgroup and $G$, $H$  the universal complexification of $G_\R$, $H_\R$ respectively. Assume that $i \colon G_\R \to G$ is injective, $G$ linear algebraic and $G/H$ affine of complex dimension $\ge 2$. Denote by $\tau \colon G/H \to G/H$ the antiholomorphic involution induced by $\tau_G$ on $G/H$. 

Then for every Hamiltonian diffeomorphism $\varphi \in \diff_{\omega_\R}(T^*_\R (G_\R/H_\R))$, for any $k \in \N$ and any positive continuous function $\epsilon$ on $T_\R ^*(G_\R/H_\R)$, there exists a holomorphic symplectomorphism $\Phi \in \aut^{\tilde{\tau}}_{\omega}(T^* (G/H))$ such that
	\[	d_k( j^k \Phi (p) - j^k \varphi(p) ) < \epsilon (p) \quad \text{ for all }\, p \in T_\R^* (G_\R/H_\R).	\]
\end{theorem}

\begin{proof} When the fixed-point set of the antiholomorphic involution $\tau$ is equal to $G_\R/H_\R$, apply Theorem \ref{thm: Carleman-Cotan}. If $G_\R/H_\R$ is a proper subset of the fixed-point set, see the remark below.
\end{proof}

\begin{remark}
    (1) We may weaken Condition (ii) $X_\R = \mathrm{Fix}(\sigma)$ of Definition \ref{definition: realform} as follows: $X_\R$ is the union of some connected components of the fixed-point set $\mathrm{Fix}(\sigma)$. This is simply because we can extend a Hamiltonian diffeomorphism of $X_\R$ to $\mathrm{Fix}(\sigma)$ by taking $\id$ on $\mathrm{Fix}(\sigma) \setminus X_\R$. 
    
    (2) It is known that if the subgroup $H$ is reductive then the space $G/H$ is always affine,
    but there is no known group-theoretic characterization for $G$ and $H$ which would say when $G/H$ is affine.

    (3) In  the case that $G/H$ is 1-dimensional, it can be only biholomorphic to $\C^*$ or to $\C$. Since $T^* \C$ and $T^* \C^*$ have the $\tilde{\tau}$-symplectic (Hamiltonian) density property, by Andrist--Huang \cite{MR4986786}*{Lemma 3.1} and \cite{MR4986786}*{Lemma 3.2, Theorem 3.5} respectively, for $\tau$ the complex conjugate, the above holomorphic approximation for Hamiltonian diffeomorphisms on them also holds. However, for $T^* \C^*$ with the involution $\tau(z) = \overline{z}^{-1}$ on $\C^*$, it seems that both the $\tilde{\tau}$-Hamiltonian density property and approximation for Hamiltonian diffeomorphisms do not hold, however we are not able to disprove them. 
\end{remark}

\begin{remark}
    In fact the algebraic cotangent bundle of smooth affine homogeneous spaces $G/H$, with $G$ linear algebraic and $H \subset G$ closed, is symplectomorphic to a closed coadjoint orbit of the semidirect product $G \ltimes W$ where $W \subset \C[G/H]$ is a finite dimensional representation of $G$ generating $\C[G/H]$ as $\C$-algebra \cite{MR1619845}*{Proposition 5.1}. Thus the Hamiltonian density property of $T^*(G/H)$ and the holomorphic approximation of Hamiltonian diffeomorphisms on real forms of $G/H$ were also considered in \cite{MR4423269} from the viewpoint of closed coadjoint orbits of complex Lie groups.
\end{remark}

\subsection{Inhomogeneous examples}

Given a polynomial $p\colon\mathbb{C}\to \mathbb{C}$ with simple zeros, we consider the variety
\[ 
    D_p =\{ (x,y,z)\in \mathbb{C}^3: xy=p(z)\} 
\]
called \textbf{Danielewski surface}. This is an algebraic manifold, since $p$ has only simple zeros. It is known that $D_p$ has the algebraic density property see \cite{MR2350038}*{Theorem 1} and also \cite{AFHKS25}*{\S 4.1}. In particular the following three complete algebraic vector fields
\begin{align*}
    V_1 = x \frac{\partial}{\partial x} - y \frac{\partial}{\partial y}, \quad
    V_2 = p'(z) \frac{\partial}{\partial x} + y \frac{\partial}{\partial z}, \quad
    V_3 = p'(z) \frac{\partial}{\partial y} + x \frac{\partial}{\partial z},
\end{align*}
span the tangent bundle of $D_p$ \cite{AFHKS25}*{Lemma 4.2}, thus $D_p$ is holomorphically flexible. 

Restricting to the case of $p$ with real coefficients and using the complex conjugation $\tau$ as the antiholomorphic involution, we see that $D_p$ is $\tau$-holomorphically flexible as the above spanning vector fields are $\tau$-compatible. 

Moreover by \cite{MR2823038}*{Proposition 2.3} and \cite{MR4883703}*{Proposition 3.1} the smooth Danielewski surface is flexible. Actually the hyperbolic vector field $V_1$ is a Lie combination of $\tau$-compatible LNDs \cite{MR2823038}*{\S 3}. Using the fact that Lie bracket of two complete vector fields can be approximated by a real linear combination of complete vector fields, we have that $V_1$ is a complex linear combination of $\tau$-compatible LNDs, hence the $\tau$-flexibility of $D_p$. 

\smallskip
In fact the polynomial $p$ can be replaced by a holomorphic function of one complex variable with simple zeros. Then the corresponding facts hold true in the holomorphic category.
Only a few Danielewski surfaces are homogeneous with respect to an action of a complex Lie group.  If $p \in \holo(\C)$ has exactly one zero, then $D_p$ is isomorphic to $\C^2$. If $p \in \holo(\C)$ has exactly two zeros, then $D_p$ is biholomorphic to the affine quadric $SL_2 (\C)/T$, where $T$ is the maximal torus.    If, however, $p$ has more than two zeros, then $D_p$ cannot be a homogeneous space of a complex Lie group. Indeed, $D_p$ has trivial fundamental group and strongly deformation-retracts onto a chain of at least two spheres, see Lind \cite[Section 3.3]{Lind2006}, so it is not on the list of complex homogeneous surfaces in \cite{MR0867961} given by Huckleberry.

\smallskip
Also the real locus, i.e.\@ the fixed-point set is not homogeneous under any real Lie group for $n\ge 3$. Indeed the 
retraction given in Lind's PhD thesis is leaving invariant the real locus and retracts it to a chain of $n-1$ one-dimensional spheres. However by Mostow \cite{MR0048464} the only surfaces on which a Lie group acts transitively  are 
 plane, sphere, cylinder, torus, projective plane, Möbius strip or Klein bottle.

\end{document}